\documentclass[A4paper,12pt]{amsart}
\usepackage{graphicx}

\newtheorem{thm}{Theorem}
\newtheorem*{thm*}{Theorem}

\newtheorem{lem}{Lemma}
\newtheorem{prop}{Proposition}
\newtheorem{rem}{Remark}

\newcommand{\ninf}{n\rightarrow\infty}

\renewcommand{\Re}{\text{Re}}

\DeclareMathOperator{\C}{\mathbb{C}}

\renewcommand{\l}{\left}
\renewcommand{\r}{\right}

\newcommand{\sumn}{\sum\limits_{n=1}^{\infty}}
\newcommand{\tr}{\text{tr}}
\newcommand{\discr}{\text{discr}}

\begin{document}

\title[One model in critical hyperbolic situation]{Spectral analysis of a class of hermitian Jacobi
matrices in a critical (double root) hyperbolic case.}

\author{Serguei Naboko}

\address{Department of Mathematical Physics, Institute of Physics, St. Petersburg University,
Ulianovskaia 1, St. Petergoff, St. Petersburg, 198904}
\email{naboko@snoopy.phys.spbu.ru}

\thanks{This work was supported by the INTAS  grant no. 05-1000008-7883 and partially by the grant RFBR-06-01-00249.}

\author{Sergey Simonov}

\address{Department of Mathematical Physics, Institute of Physics, St. Petersburg University,
Ulianovskaia 1, St. Petergoff, St. Petersburg, 198904}
\email{sergey\_simonov@mail.ru}


\date{}

\begin{abstract}
We consider a class of Jacobi matrices with periodically modulated
diagonal in a critical hyperbolic ("double root") situation. For
the model with "non-smooth" matrix entries we obtain the
asymptotics of generalized eigenvectors and analyze the spectrum.
In addition, we reformulate a very helpful theorem from
\cite{Janas-Moszynski} in its full generality in order to serve
the needs of our method.
\end{abstract}

\maketitle

\section{Introduction}
In the present paper, we consider a class of Jacobi matrices with
periodically modulated growing diagonal giving an example of the
"double root" problem. This means a critical situation for
asymptotics of the generalized eigenvectors related to the matrix.
Such a situation particularly arises in the spectral phase
transition phenomenon. If the matrix depends on some parameters,
the decomposition of its spectrum into different types (absolutely
continuous, singular continuous, pure point, discrete) may be
independent of these parameters. But if the structure of this
decomposition changes under a variation of the parameters, a
spectral phase transition occurs. If this change happens by a
jump, such a phenomenon is called a spectral phase transition of
the first type, whereas if this change is smooth with the change
of the parameters when they move across some hyper-surface in a
space, it is called a spectral phase transition of the second
type. The transition in types of spectrum is tightly related to
the change of form of asymptotics of generalized eigenvectors due
to subordinacy theory (Gilbert-Pearson \cite{Gilbert-Pearson}),
which was generalized to the case of Jacobi matrices in
\cite{Khan-Pearson}.

Last decade the spectral analysis of Jacobi matrices attracted the
attention of many specialists in operator theory and mathematical
physics.

We consider the Jacobi matrix $J$ with diagonal entries $b_n$:
\begin{equation*}
    b_n=\l\{
    \begin{array}{c}
      bn^{\alpha} \text{ for odd values of }n\\
      0 \text{ for even values of }n \\
    \end{array}
    \r.
\end{equation*}
(where $\alpha$ and $b$ are real parameters $\alpha\in(\frac23;1)$
and $b\ne0$) and off-diagonal entries (weights)
$$a_n=n^{\alpha}.$$ The model demonstrates the situation of the
spectral phase transition of the first order corresponding to the
"moment of transition" exactly.  We are interested in asymptotics
of generalized eigenvectors, i.e., solutions of the spectral
recurrence relation
\begin{equation}\label{spectral_equation}
    a_{n-1}u_{n-1}+b_nu_n+a_nu_{n+1}=\lambda u_n,\ n\geq2.
\end{equation}
This model has been studied in \cite{Damanik-Naboko} and in
particular the following result was obtained: for $\lambda<0$
there are two solutions $u^+_n$ and $u^-_n$ of
\eqref{spectral_equation} with the following asymptotics as
$\ninf$:
\begin{equation*}
    \begin{array}{l}
      u_{2n}^\pm\sim(-1)^{n}n^{-\frac{\alpha}4}\exp\l(\pm\sqrt{\frac{b\lambda}{2^{\alpha}}}\frac{n^{1-\frac{\alpha}{2}}}{1-\frac{\alpha}{2}}\r)\\
      u_{2n+1}^\pm\sim\pm\sqrt{\frac{\lambda}{2^{\alpha}b}}\l(1-\frac{\alpha}2\r)(-1)^{n}n^{-\frac{3\alpha}4}\exp\l(\pm\sqrt{\frac{b\lambda}{2^{\alpha}}}\frac{n^{1-\frac{\alpha}{2}}}{1-\frac{\alpha}{2}}\r)\\
    \end{array}
\end{equation*}
The problem of determining the asymptotics for $\lambda>0$ was
stated, and in the present paper, we show that the answer has the
same form. However this interesting question is not the main
concern of this paper. The principal difficulty in our analysis is
that the situation is "critical hyperbolic" unlike the "critical
elliptic" situation in \cite{Damanik-Naboko}, which roughly means
that exponents in the answer grow and decay if $\lambda>0$ and
oscillate if $\lambda<0$.

Let us explain the problem of the critical situation in more
detail. Whenever one deals with the three-term recurrence relation
\eqref{spectral_equation}, it is often useful to write it in the
vector form introducing the sequence $\overrightarrow{u}_n:=
\left(%
\begin{array}{c}
  u_{n-1} \\
  u_n \\
\end{array}%
\right)$ and the transfer-matrix $B_n:=
\left(%
\begin{array}{cc}
  0 & 1 \\
  -\frac{a_{n-1}}{a_n} & \frac{\lambda-b_n}{a_n} \\
\end{array}%
\right)$. So \eqref{spectral_equation} is equivalent to the
discrete linear system in $\mathbb{C}^2$:
\begin{equation}\label{vector form}
    \overrightarrow{u}_{n+1}=B_n\overrightarrow{u}_n,\ n\geq2.
\end{equation}
The solution for such a system is obtained by taking the
chronological product of transfer matrices,
$\l(\prod\limits_{k=2}^{n}B_k\r)
\overrightarrow{u}_2=B_n\cdot...\cdot B_2 \overrightarrow{u}_2$.
The analysis becomes much easier if the matrix of the system is in
some sense smooth-in-$n$ (say, has a limit and asymptotic
expansion in inverse powers of $n$ as $n\rightarrow\infty$). In
our case the transfer-matrix is not smooth in this sense because
the coefficients of the spectral equation "jump" all the time.
This is the first (simple) problem that we face and it may be
solved by taking the product of two consecutive transfer-matrices,
introducing the new linear system with the sequence of the
coefficient matrices
\begin{equation}\label{matrix_M}
M_n:=B_{2n}B_{2n-1}
\end{equation}
which turns out to be smooth-in-$n$. It is easy to see that the
sequence $M_n$ has a limit
$M:=\lim\limits_{n\rightarrow\infty}M_n$ with $\det M=1$. The are
three possibilities for the eigenvalues of $M$: they can be
\begin{itemize}
    \item unimodular and complex conjugate (the elliptic
    situation),
    \item real and different (the hyperbolic situation; hence
    one of them is greater than $1$ and another less than $1$
    in absolute value), or
    \item coincide and equal $1$ or $-1$ (the critical situation=the double root
    case).
\end{itemize}
In the hyperbolic situation solutions are supposed to grow or
decay, in the elliptic situation they are supposed to oscillate
having similar behavior of their norms $\|\overrightarrow{u}_n\|$.
The numerous variants of analogues to Levinson Theorem (also known
as Bernzaid-Lutz Theorem \cite{Bernzaid-Lutz}) for differential
linear systems \cite{Coddington-Levinson} can be applied in these
two cases, cf. \cite{Janas-Moszynski}, \cite{Silva uniform}.

In the critical situation one can also distinguish the "critical
elliptic" and the "critical hyperbolic" cases. This separation
depends on the lower orders behavior of $M_n$ as $\ninf$, namely
on the asymptotic sign of the discriminant ($\discr M_n:=(\tr
M_n)^2-4\det M_n$) of matrices $M_n$. In the critical case, the
matrix $M$ is similar to the Jordan block, with powers
\begin{equation*}
\left(%
\begin{array}{cc}
  1 & 1 \\
  0 & 1 \\
\end{array}%
\right)^n
=\left(%
\begin{array}{cc}
  1 & n \\
  0 & 1 \\
\end{array}%
\right),
\end{equation*}
and the difficulty here is not that the powers of the matrix grow
with $n$ (as they do in the hyperbolic situation), but the fact
that the large entry is off-diagonal. This mixes upper and lower
components of the solution, which makes the system unstable and
sensitive to small perturbations. The problem of the "critical
hyperbolic" situation was considered for smooth matrix elements in
\cite{JN-Sheronova}. In the present paper, considering a model of
Jacobi matrix with "oscillating" diagonal, we intend to simplify
greatly the scheme of successive transformations of the matrix
system which was used there by making it more general and
transparent (Section \ref{calculations}). This approach differs
from that of \cite{Damanik-Naboko} for "critical elliptic"
situation, because in our case we have to deal with growing
exponents (see Step 3 on page \pageref{step3}). Moreover, we
consider in the appendix a theorem from \cite{Janas-Moszynski}
which is necessary for the final step of our method. The problem
is that the original formulation of this theorem states
asymptotics of only one (principal) solution. Although it is
surely the most difficult part of the problem considered, we want
to state the result in its full generality.

The method of the present paper works for a much wider class of
Jacobi matrices. However our goal is to present a simple
formulation of the method and show by means of an example how it
works.

A similar problem (the critical hyperbolic situation) was
considered in \cite{Janas} for a "smooth" model where the author
used a completely different method related to Kelley's paper
\cite{Kelley}.

\section{Preliminaries}
As usual, the operator $J$ in the Hilbert space $l^2(\mathbb{N})$
is first defined (as $\mathcal{J}$) on the linear set of vectors
which have only a finite number of non-zero components,
$l_{fin}(\mathbb{N})$, by the rule
\begin{equation*}
    \begin{array}{l}
      (\mathcal{J}u)_1=a_1u_1+b_1u_2 \\
      (\mathcal{J}u)_n=b_{n-1}u_{n-1}+a_nu_n+b_nu_{n+1},\text{ for }n\geq2. \\
    \end{array}
\end{equation*}
Then its closure $J=\overline{\mathcal{J}}$ is a self-adjoint
operator provided the Carleman condition
$\sum\limits_{n=0}^{\infty}\frac1{a_n}=+\infty$ \cite{Berezanskii}
is satisfied. In the standard basis $\{e_n\}_{n=1}^{\infty}$
(where $e_n$ is the vector with all the components zeros except
the $n$-th) the operator $J$ admits the following matrix
representation:
\begin{equation*}
    \left(%
    \begin{array}{cccc}
    b_1 & a_1 & 0 & \cdots \\
    a_1 & b_2 & a_2 & \cdots \\
    0 & a_2 & b_3 & \cdots \\
    \vdots & \vdots & \vdots & \ddots \\
    \end{array}%
    \right).
\end{equation*}

We start with the system
\begin{equation}\label{main_system}
\left(
    \begin{array}{c}
    u_{2n} \\
    u_{2n+1}
    \end{array}
\right)=M_n \left(
    \begin{array}{c}
    u_{2n-2} \\
    u_{2n-1}
    \end{array}
\right)
\end{equation}
(see \eqref{vector form}, \eqref{matrix_M}). One can obtain a
solution of this system directly, by taking a product of matrices
$M_n$: $\left(
    \begin{array}{c}
    u_{2n} \\
    u_{2n+1}
    \end{array}
\right)=\l[\prod\limits_{k=2}^nM_k\r] \left(
    \begin{array}{c}
    u_{2} \\
    u_{3}
    \end{array}
\right)$. Matrices $M_n$ are:
\begin{multline*}
    M_n=B_{2n}B_{2n-1}
    \\
        =
        \left(
            \begin{array}{cc}
            -1 & -b \\
            0 & -1
            \end{array}
        \right)
        +\frac{\lambda}{(2n)^\alpha}
        \left(
            \begin{array}{rc}
            0 & 1 \\
            -1 & -b
            \end{array}
        \right)
        +\frac{\alpha}{2n}I+O\l(\frac1{n^{2\alpha}}\r)
\end{multline*}
having smooth-in-$n$ asymptotics as $\ninf$.

The eigenvalues of the limit matrix $M:=
\left(%
\begin{array}{cc}
  -1 & -b \\
  0 & -1 \\
\end{array}%
\right)$ coincide, so the situation is critical (the double root
case). By an easy calculation, one can see that discriminants of
matrices $M_n$ equal $\discr
M_n=\frac{4b\lambda}{(2n)^{\alpha}}+O\l(\frac1n\r)$, hence indeed
$\lambda<0$ corresponds to the elliptic situation and $\lambda>0$
corresponds to the hyperbolic situation (both critical).

Our method is based upon a sequence of transformations which are
determined by some anzats. Similar consideration for the anzats
constructing may be also found in \cite{JN-Sheronova}.

\begin{rem}
    In what follows, we use transformations that are in fact discrete analogues to variation of parameters method transformations.
    Whenever one deals with the product of matrices, say $A_n$, one can as well consider
    matrices $C_n=T_{n+1}^{-1}A_nT_n$ (where the matrix sequence $T_n$ is chosen
    somehow; we will call matrices $T_n$ affinity-like). Due to the
    cancellation of intermediate terms the product of matrices
    $A_n$ equals
    \begin{equation*}
        \prod_{n=n_1}^{n_2}A_n=\prod_{n=n_1}^{n_2}(T_{n+1}C_nT_n^{-1})=T_{n_2+1}\l(\prod_{n=n_1}^{n_2}C_n\r)T_{n_1}^{-1}.
    \end{equation*}
    So the study of the linear difference system with coefficient matrices $A_n$
    can be completely reduced to the study of the linear difference
    system with coefficient matrices $C_n$.
\end{rem}

\section{Calculation of the asymptotics in the hyperbolic case.}\label{calculations}
In this section, we proceed through several transformations in
order to simplify the problem and finally obtain the system which
can be treated with Janas-Moszynski Theorem (which is properly
adjusted in the appendix). So we divide this section into four
steps. In fact we have already made a "zero" step in the first
section, which is reduction to the smooth matrix system by taking
the product of two transfer-matrices. But we do not include this
step into Section \ref{calculations} because the double root
problem actually arises only at this stage. Moreover after
grouping the transfer-matrices by pairs one does not need to
perform any inverse transformation in order to obtain the answer.

\underline{Step 1: Reduction of the meaningful part of $M_n$ to
the transfer-}
\\\underline{matrix form.}

We write the spectral equation in matrix form to settle the
problem of periodically modulated coefficients by grouping
transfer-matrices in pairs. But now it is simpler to consider a
smooth three-term recurrence relation which is equivalent to the
system. The transfer-matrix corresponding to the three-term
recurrence relation should have entries $0$ and $1$ in the upper
row. In this step, we find the proper transformation which makes
the coefficient matrix of the system resemble a transfer-matrix.
This transformation is generated by the matrix
sequence $T_n$ of the form
\begin{equation*}
    T_n:=(-1)^n\l[
    \left(
        \begin{array}{rr}
        1 & -b \\
        1 & 0
        \end{array}
    \right)
    +\frac{\lambda}{(2n)^\alpha}
    \left(
        \begin{array}{cc}
        b+\frac{1}{2b} & 0 \\
        \frac{1}{2b} & -\frac12
        \end{array}
    \right)
    +\frac{\alpha}{2n}
    \left(
        \begin{array}{rc}
        0 & 0 \\
        -1 & b
        \end{array}
    \right)\r]
    .
\end{equation*}
The transformation
\begin{equation}\label{N_n}
    N_n:=T_{n+1}M_n T_n^{-1}
\end{equation}
gives
\begin{equation*}
    N_n=
    \left(
        \begin{array}{rc}
        0 & 1 \\
        -1 & 2
        \end{array}
    \right)
    +\frac{b\lambda}{(2n)^\alpha}
    \left(
        \begin{array}{cc}
        0 & 0 \\
        0 & 1
        \end{array}
    \right)
    +\frac{\alpha}{n}
    \left(
        \begin{array}{cr}
        0 & 0 \\
        1 & -1
        \end{array}
    \right)
    +O\l(\frac1{n^{2\alpha}}\r).
\end{equation*}
The meaningful part (or the part which we expect to be meaningful
- the sum of first few terms) is now a matrix with the upper row
of the form $0$, $1$. The exact form of matrices $T_n$ can be
determined from this requirement (but the answer is not unique of
course). Indeed, try to find $T_n$ in the form
\begin{equation}\label{T_n}
T_n=(-1)^n\l[
\left(%
\begin{array}{cc}
  1 & -b \\
  1 & 0 \\
\end{array}%
\right)+\frac{\lambda}{(2n)^{\alpha}}T^{(1)}+\frac{\alpha}{2n}T^{(2)}\r],
\end{equation}
with unknown matrices $T^{(1)}$ and
$T^{(2)}$ independent of $n$. Denote $T:=\left(%
\begin{array}{cc}
  1 & -b \\
  1 & 0 \\
\end{array}%
\right)$ which is chosen in order to satisfy
$$
-T\left(%
\begin{array}{cc}
  -1 & -b \\
  0 & -1 \\
\end{array}%
\right)T^{-1}=\left(%
\begin{array}{cc}
  0 & 1 \\
  -1 & 2 \\
\end{array}%
\right) =:N.
$$
By substitution of $T_n$ in the form \eqref{T_n} into the relation
$T_{n+1}M_n=N_nT_n$, looking at the terms of orders
$\frac1{n^\alpha}$ and $\frac1{n}$ one obtains the following
linear conditions on the matrices $T^{(1)}$ and $T^{(2)}$:
\begin{equation*}
\begin{array}{l}
    [T^{(1)}T^{-1},
    \left(%
        \begin{array}{cc}
        0 & 1 \\
        -1 & 2 \\
        \end{array}%
    \right)
    ]=T
    \left(%
        \begin{array}{cc}
        0 & 1 \\
        -1 & -b \\
        \end{array}%
    \right)
    T^{-1}+
    \left(%
        \begin{array}{cc}
        0 & 0\\
        * & * \\
        \end{array}%
    \right)
\\{}[T^{(2)}T^{-1},
    \left(%
        \begin{array}{cc}
        0 & 1 \\
        -1 & 2 \\
        \end{array}%
    \right)
    ]=I+
    \left(%
        \begin{array}{cc}
        0 & 0\\
        * & * \\
        \end{array}%
    \right).
\end{array}
\end{equation*}
By stars we denote matrix entries which are allowed to be
non-zero. Therefore the problem
may be reduced to the following one. To prove for the commutator
equation
\begin{equation*}
    [X,
    \left(%
        \begin{array}{cc}
        0 & 1 \\
        -1 & 2 \\
        \end{array}%
    \right)]=
    \left(%
        \begin{array}{cc}
        f_1 & f_2 \\
        x_1 & x_2 \\
        \end{array}%
    \right),
\end{equation*}
that for any given values $f_1$ and $f_2$ there exist unique
values $x_1$ and $x_2$ and $2\times2$ matrix $X$ (obviously not
unique) which satisfy the equation. Multiplying the equality on
the left by the matrix
$\left(%
\begin{array}{cc}
  1 & 0 \\
  -1 & 1 \\
\end{array}%
\right)$ and on the right by
$\left(%
\begin{array}{cc}
  1 & 0 \\
  1 & 1 \\
\end{array}%
\right)$ (its inverse) we obtain another form of the commutator
equation:
\begin{multline*}
    [Y,
    \left(%
        \begin{array}{cc}
        1 & 1 \\
        0 & 1 \\
        \end{array}%
    \right)]=[Y,
    \left(%
        \begin{array}{cc}
        0 & 1 \\
        0 & 0 \\
        \end{array}%
    \right)]
    \\
        =\left(%
            \begin{array}{cc}
            -y_3 & y_1-y_4 \\
            0 & y_3 \\
            \end{array}%
        \right)=
        \left(%
            \begin{array}{cc}
            f_1+f_2 & f_2 \\
            x_1+x_2-f_1-f_2 & x_2-f_2 \\
            \end{array}%
        \right).
\end{multline*}
We denoted
$Y:=\left(%
\begin{array}{cc}
  y_1 & y_2 \\
  y_3 & y_4 \\
\end{array}%
\right):=
\left(%
\begin{array}{cc}
  1 & 0 \\
  -1 & 1 \\
\end{array}%
\right)X
\left(%
\begin{array}{cc}
  1 & 0 \\
  1 & 1 \\
\end{array}%
\right)$. It follows now that $x_2=-f_1$ and $x_1=2f_1+f_2$ are
uniquely determined by $f_1$ and $f_2$, and the matrix $Y$ exists
and is unique up to $c_1I+c_2
\left(%
\begin{array}{cc}
  0 & 1 \\
  0 & 0 \\
\end{array}%
\right)$ with any $c_1$ and $c_2$.

Therefore one can reduce the original system \eqref{main_system}
to the new one, where $M_n$ are replaced by $N_n$. Note that the
exact form of the matrices $T^{(1)}$ and $T^{(2)}$ is not
essential for the result of transformation. We need only to prove
the existence of such matrices.

\underline{Step 2: Reduction of the main term to the identity
matrix.}

The structure of the main part of the system allows to write
\begin{equation*}
    N_n=
    \left(%
        \begin{array}{cc}
        0 & 1 \\
        -F_2(n) & -F_1(n) \\
        \end{array}%
    \right)
    +O\l(\frac1{n^{2\alpha}}\r),
\end{equation*}
where the error term is a $2\times2$ matrix which norm is
$O(n^{-2\alpha})$ and the matrix entries $F_1(n)$ and $F_1(n)$ are
(denote $B:=\sqrt{\frac{b\lambda}{2^{\alpha}}}$)
\begin{equation*}
    F_1(n)=-2-\frac{B^2}{n^{\alpha}}+\frac{\alpha}n,\ F_2(n)=1-\frac{\alpha}n.
\end{equation*}
Our goal is to use the fact that the system
\begin{equation*}
    \left(%
        \begin{array}{c}
        v_{n+1} \\
        w_{n+1} \\
        \end{array}%
    \right)
    =
    \left(%
        \begin{array}{cc}
        0 & 1 \\
        -F_2(n) & -F_1(n) \\
        \end{array}%
    \right)
    \left(%
        \begin{array}{c}
        v_n \\
        w_n \\
        \end{array}%
    \right)
\end{equation*}
where the remainder is omitted, is \emph{equivalent} to the
three-term recurrence relation
\begin{equation}\label{recurrence ralation}
    u_{n+1}+F_1(n)u_n+F_2(n)u_{n-1}=0.
\end{equation}
The latter has two "approximate solutions" of the form
\begin{equation*}
    z^{\pm}_n=n^\gamma e^{\pm An^\delta},
\end{equation*}
with $\gamma=-\frac{\alpha}{4}$, $\delta=1-\frac{\alpha}{2}$,
$A=\frac{B}{\delta}$. This means that
\begin{equation}\label{approximate_solution}
    z^{\pm}_{n+1}+F_1(n)z^{\pm}_n+F_2(n)z^{\pm}_{n-1}=O(n^{-2\alpha})z^{\pm}_n
\end{equation}
and can be verified by a direct calculation. The form of these
"approximate solutions" can be obtained, for instance, by analogy
with the WKB method (see \cite{JN-Sheronova} for that type of
argument). Having this structure of the "solution" one can
determine unknown values of $\gamma$, $\delta$ and $A$ from the
condition of cancellation of all decreasing terms in
\eqref{approximate_solution} up to the order $O(n^{-2\alpha})$.
Equation \eqref{approximate_solution} implies that
\begin{multline}\label{vector_form}
    \left(%
        \begin{array}{c}
        z^{\pm}_n \\
        z^{\pm}_{n+1} \\
        \end{array}%
    \right)
    =
    \l(
    \left(%
        \begin{array}{cc}
        0 & 1 \\
        -F_2(n) & -F_1(n) \\
        \end{array}%
    \right)
    +
    \left(%
        \begin{array}{cc}
        0 & 0 \\
        0 & O(n^{-2\alpha}) \\
        \end{array}%
    \right)\r)
    \left(%
        \begin{array}{c}
        z^{\pm}_{n-1} \\
        z^{\pm}_n \\
        \end{array}%
    \right)
    \\
        =(N_n+O(n^{-2\alpha}))
        \left(%
        \begin{array}{c}
        z^{\pm}_{n-1} \\
        z^{\pm}_n \\
        \end{array}%
    \right).
\end{multline}
It is useful now to write the last vector equality in a matrix
form. Denote
\begin{equation*}
    S_n=
    \left(
        \begin{array}{cc}
        z^-_{n-1} & z^+_{n-1} \\
        z^-_n & z^+_n
        \end{array}
    \right),
\end{equation*}
then combining \eqref{vector_form} for both signs one has
\begin{equation*}
    S_{n+1}=(N_n+O(n^{-2\alpha}))S_n
\end{equation*}
and hence
\begin{equation*}
    S_{n+1}^{-1}N_nS_n=I+S_{n+1}^{-1}O(n^{-2\alpha})S_n.
\end{equation*}
This actually follows \emph{only} from the fact that $z_n^{\pm}$
are "approximate solutions" of the recurrence relation
\eqref{recurrence ralation}, i.e., from
\eqref{approximate_solution}.

Substituting the expression for $z^{\pm}_n=n^\gamma e^{\pm
An^\delta}$ into the term
$$
S_{n+1}^{-1}O(n^{-2\alpha})S_n=\frac1{\det S_{n+1}}
    \left(%
        \begin{array}{cc}
        z_n^+z_n^-O(n^{-2\alpha}) & z_n^{+2}O(n^{-2\alpha}) \\
        z_n^{-2}O(n^{-2\alpha}) & z_n^+z_n^-O(n^{-2\alpha}) \\
        \end{array}%
    \right)
$$
and using the fact that $\det S_{n+1}\sim2A\delta n^{-\alpha}$ as
$\ninf$ (for calculations, cf. \cite{Damanik-Naboko}; one can also
find this rate of decay from the fact of modified wronskian
persistence), one obtains
\begin{equation}\label{bad_matrix}
    S_{n+1}^{-1}N_nS_{n}=I+\frac1{n^{3\alpha/2}}
    \left(%
        \begin{array}{cc}
        O(1) & e^{2An^{\delta}}O(1) \\
        e^{-2An^{\delta}}O(1) & O(1) \\
        \end{array}%
    \right)=:K_n.
\end{equation}

These calculations enable us to reduce the original system to the
system with the coefficient matrix $K_n$, for which the "main
term" is identity matrix. The problem is that it contains
exponentially increasing anti-diagonal terms. In the next step we
show how to overcome this difficulty due to the symmetric
cancellation of both anti-diagonal entries.

\begin{rem}
The above calculations demonstrate a simple "geometrical" approach
based on reduction to the tree-term recurrence relation
\eqref{recurrence ralation}. Another approach (a simplification of
calculations from \cite{JN-Sheronova}) may be rather
straightforward: the "geometrical construction" is replaced by an
explicit calculation. Since $\det
S_{n+1}=z_n^-z_{n+1}^+-z_n^+z_{n+1}^-$, the substitution of
matrices $S_n$ gives:
\begin{multline*}
    S_{n+1}^{-1}
    \left(%
        \begin{array}{cc}
        0 & 1 \\
        -F_2(n) & -F_1(n) \\
        \end{array}%
    \right) S_n=\frac1{z_n^-z_{n+1}^+-z_n^+z_{n+1}^-}
    \\
        \times\left(%
            \begin{array}{cc}
            z_n^-z_{n+1}^++z_n^+(F_1z_n^-+F_2z_{n-1}^-) & z_n^+(z_{n+1}^++F_1z_n^++F_2z_{n-1}^+) \\
            -z_n^-(z_{n+1}^-+F_1z_n^-+F_2z_{n-1}^-) & -z_n^+z_{n+1}^--z_n^-(F_1z_n^++F_2z_{n-1}^+) \\
            \end{array}%
        \right).
\end{multline*}
After adding and subtracting terms $z_n^+z_{n+1}^-$ in the
upper-left entry and $z_n^-z_{n+1}^+$ in the lower-right entry for
extracting the determinant, the last expression becomes the
following one:
\begin{multline*}
    I+\frac1{z_n^-z_{n+1}^+-z_n^+z_{n+1}^-}
    \\
        \times\left(%
            \begin{array}{cc}
            z_n^+(z_{n+1}^-+F_1z_n^-+F_2z_{n-1}^-) & z_n^+(z_{n+1}^++F_1z_n^++F_2z_{n-1}^+) \\
            -z_n^-(z_{n+1}^-+F_1z_n^-+F_2z_{n-1}^-) & -z_n^-(z_{n+1}^-+F_1z_n^++F_2z_{n-1}^+) \\
            \end{array}%
        \right).
\end{multline*}
It is remarkable that the expressions
$z_{n+1}^{\pm}+F_1z_n^{\pm}+F_2z_{n-1}^{\pm}$ appear in each
matrix entry of the second term. Now substituting
\eqref{approximate_solution} and taking into consideration that
$\det S_{n+1}\sim2A\delta n^{-\alpha}$, one obtains the same
expression as in \eqref{bad_matrix}:
\begin{multline*}
    I+\frac1{\det S_{n+1}}
    \left(%
        \begin{array}{cc}
        z_n^+z_n^-O(n^{-2\alpha}) & z_n^{+2}O(n^{-2\alpha}) \\
        z_n^{-2}O(n^{-2\alpha}) & z_n^+z_n^-O(n^{-2\alpha}) \\
        \end{array}%
    \right)
    \\
        =I+\frac1{n^{3\alpha/2}}
        \left(%
            \begin{array}{cc}
            O(1) & e^{2An^{\delta}}O(1) \\
            e^{-2An^{\delta}}O(1) & O(1) \\
            \end{array}%
        \right).
\end{multline*}
\end{rem}
The problem now is the growing exponent in the upper-right entry
of the matrix, which leaves no chance to regard the remainder
small in any sense. This exponent can be compensated by the
decaying one in the lower-left entry as we show below.

\underline{Step 3: Elimination of exponentially increasing
off-diagonal terms.}\label{step3}

In order to produce the elimination we perform yet another
transformation with affinity-like matrix
\begin{equation*}
    X_n=
    \left(%
    \begin{array}{cc}
    e^{2An^{\delta}} & 0 \\
    0 & 1 \\
    \end{array}%
    \right),
\end{equation*}
which yields
\begin{multline}\label{last_matrix}
    L_n:=X_{n+1}^{-1}K_nX_n=
    \left(%
        \begin{array}{cc}
        e^{2A(n^{\delta}-(n+1)^{\delta})} & 0 \\
        0 & 1 \\
        \end{array}%
    \right)
    +O\l(\frac1{n^{3\alpha/2}}\r)
    \\
        =\left(%
            \begin{array}{cc}
            1-\frac{2A\delta}{n^{\alpha/2}}+\frac{(2A\delta)^2}{2n^{\alpha}} & 0 \\
            0 & 1 \\
            \end{array}%
        \right)+O\l(\frac1{n^{3\alpha/2}}\r).
\end{multline}

As a result of all the transformations, the original system is
reduced to a new one with coefficient matrices $L_n$ for which
Janas-Moszynski Theorem \cite{Janas-Moszynski} (Theorem \ref{JM
lemma}, Appendix) is applicable. Take
$p_n=\frac{2A\delta}{n^{\alpha/2}}$, $V_n\equiv V=
\left(%
\begin{array}{cc}
  -1 & 0 \\
  0 & 0 \\
\end{array}%
\right)$ and let $R_n=O(n^{-3\alpha/2})$ be the matrix remainder
which belongs to $l^1$ for $\alpha>2/3$. The theorem asserts that
the system with coefficient matrices $L_n$ has a basis of
solutions of the form
$e^{-2An^{\delta}}\l(\overrightarrow{e}_1+o(1)\r)$ and
$\overrightarrow{e}_2+o(1)$ (we use notations
$\overrightarrow{e}_1=
\left(%
    \begin{array}{c}
    1 \\
    0 \\
    \end{array}%
\right)$, $\overrightarrow{e}_2=\left(%
    \begin{array}{c}
    1 \\
    0 \\
    \end{array}%
\right)$).

\begin{rem}
Note that we need to apply Theorem \ref{JM lemma} only to ignore
the remainder term of order $O(n^{-3\alpha/2})$. Without it, the
system generated by matrix \eqref{last_matrix} obviously has two
exact solutions $e^{-2An^{\delta}}\overrightarrow{e}_1$ and
$\overrightarrow{e}_2$. However this is not entirely obvious,
since the eigenvalues of the limit matrix $\lim\limits_{\ninf}L_n$
coincide.
\end{rem}

\underline{Step 4: Returning to the original system.}

Returning to the system with coefficient matrices $M_n$ we recall
steps 1, 2 and 3. We should take into account the affinity-like
matrices that stay in front of the product of matrices $L_n$.
Since $$M_k=T_{k+1}^{-1}S_{k+1}X_{k+1}L_kX_k^{-1}S_k^{-1}T_k,$$
one has (for every $\lambda$ we choose a number $n_0$ such that
$\det T_n\neq0$ for $n\ge n_0$):
\begin{multline*}
    \left(%
        \begin{array}{c}
        u_{2n} \\
        u_{2n+1} \\
        \end{array}%
    \right)=
    \l(\prod_{k=n_0}^{n}M_k\r)
    \left(%
        \begin{array}{c}
        u_{n_0} \\
        u_{n_0+1} \\
        \end{array}%
    \right)
    \\
        =(T_{n+1}^{-1}S_{n+1}X_{n+1})\l(\prod_{k=n_0}^{n}L_k\r)X_{n_0}^{-1}S_{n_0}^{-1}T_{n_0}
        \left(%
            \begin{array}{c}
            u_{n_0} \\
            u_{n_0+1} \\
            \end{array}%
        \right)
        \\
            =(-1)^{n+1}
            \left(%
                \begin{array}{cc}
                z_n^+(1+o(1)) & z_n^+(1+o(1)) \\
                -\frac{A\delta z_n^+}{bn^{\alpha/2}}(1+o(1)) & \frac{A\delta z_n^+}{bn^{\alpha/2}}(1+o(1)) \\
                \end{array}%
            \right)
            \\
                \times\l(\prod_{k=n_0}^{n}L_k\r)
                X_{n_0}^{-1}S_{n_0}^{-1}T_{n_0}
                \left(%
                    \begin{array}{c}
                    u_{n_0} \\
                    u_{n_0+1} \\
                    \end{array}%
                \right).
\end{multline*}

As one has from the previous step, the vector
\\$\l(\prod\limits_{k=n_0}^{n}L_k\r)X_{n_0}^{-1}S_{n_0}^{-1}T_{n_0}
\left(%
    \begin{array}{c}
    u_{n_0} \\
    u_{n_0+1} \\
    \end{array}%
\right)$ is a linear combination of $e^{-2An^{\delta}}(
\overrightarrow{e}_1+o(1))$ and $(
\overrightarrow{e}_2+o(1))$. Therefore the vector $\left(%
\begin{array}{c}
  u_{2n} \\
  u_{2n+1} \\
\end{array}%
\right)$ is a linear combination of $(-1)^n
\left(%
    \begin{array}{c}
    z_n^{\pm}(1+o(1)) \\
    \pm\frac{A\delta z_n^{\pm}}{bn^{\alpha/2}}(1+o(1)) \\
    \end{array}%
\right)$. So the following result holds true.

\begin{thm}
    For any $\lambda>0$ spectral equation \eqref{spectral_equation} has a basis of
    solutions $u^+_n$ and $u^-_n$ with asymptotics as $\ninf$ of even
    components
    $$
    u_{2n}^\pm\sim
    (-1)^{n}n^{-\frac{\alpha}4}\exp\l(\pm\sqrt{\frac{b\lambda}{2^{\alpha}}}
    \frac{n^{1-\frac{\alpha}{2}}}{1-\frac{\alpha}{2}}\r)
    $$
    and asymptotics of odd components
    $$
    u_{2n+1}^\pm\sim\pm\sqrt{\frac{\lambda}{2^{\alpha}b}}\l(1-\frac{\alpha}2\r)(-1)^{n}n^{-\frac{3\alpha}4}\exp
    \l(\pm\sqrt{\frac{b\lambda}{2^{\alpha}}}\frac{n^{1-\frac{\alpha}{2}}}{1-\frac{\alpha}{2}}\r).
    $$
\end{thm}

Regarding the operator $J$ this means (due to the subordinacy
theory of Gilbert and Pearson \cite{Gilbert-Pearson}
\cite{Khan-Pearson}), that the spectrum of the operator on the
positive semiaxis is of pure point type. Moreover, if
$\lambda\in\sigma_p(J)$, then the solution $u_n^-$ is an
eigenvector of $J$.

Note that the asymptotics of $u_n^{\pm}$ as $\ninf$ for
$\lambda>0$ formally coincide with the ones obtained in
\cite{Damanik-Naboko} for $\lambda<0$ as it was mentioned above.

\section{Appendix: Janas-Moszynski Theorem revisited.}
Let us turn our attention to a theorem from
\cite{Janas-Moszynski}, which was not formulated in its full
generality there. By using a complicated technique, the existence
of the principal ("smaller") solution for some system was proved.
But the existence of the second ("larger") solution was not
stated. In what follows we prove this fact.

\begin{rem}
For the sequences $\{a_n\}$, in order to avoid any non-essential
problems related to vanishing of elements of the sequence, we
assume the notation for the product $\prod\limits_{n=1}^{N} a_n$
meaning $\prod\limits_{n:\,a_n\ne0} a_n$. We also use the notation
$l^1$ and $D^1$ for matrix sequences, i.e., the sequence of
matrices $\{M_n\}_{n=1}^{\infty}$ belongs to
\begin{itemize}
    \item $l^1$, iff $\sum\limits_{n=1}^{\infty}\|M_n\|<\infty$,
    \item $D^1$, iff $\sum\limits_{n=1}^{\infty}\|M_{n+1}-M_n\|<\infty$.
\end{itemize}
\end{rem}

Consider the linear difference system in $\C^2$
\begin{equation}\label{system}
    \left(%
        \begin{array}{c}
        u_{n+1} \\
        v_{n+1} \\
        \end{array}%
    \right)
    =
    (I+p_nV_n+R_n)
    \left(%
        \begin{array}{c}
        u_n \\
        v_n \\
        \end{array}%
    \right).
\end{equation}
Suppose it is non-degenerate, i.e., $\det(I+p_nV_n+R_n)\ne0$ for
every $n$ (this means that the system has two linearly independent
solutions).

\begin{thm}\label{JM lemma}(see \cite{Janas-Moszynski} where the the essential part of the theorem
was proved).

    Let
    \begin{itemize}
        \item $p_n\rightarrow0$ be a \emph{positive} sequence such that $\sum\limits_{n=1}^{\infty}p_n=+\infty$,
        \item $\{R_n\}\in l^1$ be a $2\times2$ matrix sequence,
        \item $\{V_n\}\in D^1$ be a \emph{real} $2\times2$ matrix sequence with
        $\rm{discr}\l(\lim\limits_{\ninf}V_n\r)\ne0$(i.e.,
        $\lim\limits_{\ninf}V_n$ has two different eigenvalues).
    \end{itemize}
    Then the system \eqref{system} has a basis of solutions
    $\overrightarrow{u}^{(1)}_n$ and $\overrightarrow{u}^{(2)}_n$ with
    the following asymptotics as $\ninf$:
    \begin{equation*}
        \overrightarrow{u}_n^{(1,2)}=\l(\prod\limits_{k=1}^n[1+p_k\mu_{1,2}(k)]\r)(\overrightarrow{x}_{1,2}+o(1)),
    \end{equation*}
    where $\mu_1$ and $\mu_2$ ($\Re\ \mu_1\leq\Re\ \mu_2$) are the
    eigenvalues of matrix $V:=\lim\limits_{\ninf}V_n$, $\overrightarrow{x}_1$ and
    $\overrightarrow{x}_2$ are the corresponding eigenvectors, $\mu_1(n)$
    and $\mu_2(n)$ are the eigenvalues of matrices $V_n$ (chosen in
    a way such that $\mu_1(n)\rightarrow\mu_1$ and
    $\mu_2(n)\rightarrow\mu_2$ as $\ninf$).
\end{thm}

\begin{proof}
The elliptic case of $\discr V<0$ (a relatively straightforward
one) is a special case of Janas-Moszynski Theorem proved in
\cite{Janas-Moszynski}. In the hyperbolic case of $\discr V>0$,
the existence of the "smaller" solution
$\overrightarrow{u}_n^{(1)}$ is also guaranteed by the named
theorem. So we need to prove only the existence of the second
("larger") solution $\overrightarrow{u}_n^{(2)}$ (the solution
corresponding to the eigenvalue $\mu_2$ with the largest real
part) in hyperbolic case. We emphasize that we do not give a new
proof of the result from \cite{Janas-Moszynski}, but only add one
extra (and rather simple) assertion to it.

Let us reduce the situation to its simpler subcase, i.e., the
system of the special form:
\begin{equation}\label{simple system}
    \left(%
        \begin{array}{c}
        u_{n+1} \\
        v_{n+1} \\
        \end{array}%
    \right)
    =
    \l(I-
    \left(%
        \begin{array}{cc}
        p_n & 0 \\
        0 & 0 \\
        \end{array}%
    \right)
    +R_n\r)
    \left(%
        \begin{array}{c}
        u_n \\
        v_n \\
        \end{array}%
    \right).
\end{equation}
For such a system, $V_n\equiv V=
\left(%
\begin{array}{cc}
  -1 & 0 \\
  0 & 0 \\
\end{array}%
\right)$.

\begin{rem}
    In fact we deal with a system of this type in Section
    \ref{calculations}, Step 3.
\end{rem}

The following statement holds:

\begin{lem}\label{new lem}
    Let
    \begin{itemize}
        \item $p_n\rightarrow0$ be a positive for sufficiently large $n$ sequence such that $\sumn p_n=+\infty$.
        \item $\{R_n\}\in l^1$ be a $2\times2$ matrix sequence.
    \end{itemize}
    Then there exists a solution to the system \eqref{simple
    system}of the form
    $$
    \left(%
        \begin{array}{c}
        u_n \\
        v_n \\
        \end{array}%
    \right)
    =\overrightarrow{e}_2+o(1)
    ,\
    \overrightarrow{e}_2=
    \left(%
        \begin{array}{c}
        0 \\
        1 \\
        \end{array}%
    \right)
    $$
    as $n\rightarrow\infty$.
\end{lem}

\begin{proof}
Without loss of generality one can assume that $0<p_n<1$ for every
$n$. Then for any two natural numbers $n_1<n_2$,
\begin{multline*}
    \l\|\prod_{n=n_1}^{n_2}\l[I-
    \left(%
        \begin{array}{cc}
        p_n & 0 \\
        0 & 0 \\
        \end{array}%
    \right)
    +R_n\r]\r\|\le
    \prod_{n=n_1}^{n_2}\l[\l\|I-
    \left(%
        \begin{array}{cc}
        p_n & 0 \\
        0 & 0 \\
        \end{array}%
    \right)
    \r\|
    +\|R_n\|\r]
    \\
        \le\prod_{n=n_1}^{n_2}[1+\|R_n\|]<\infty.
\end{multline*}
Therefore, every solution of the system \eqref{simple system} is
bounded and there exists a universal constant $C$ such that for
any natural numbers $n_1<n_2$ and any solution $u_n$
\begin{equation}\label{eq solutions boundedness}
    \l\|\l(
    \begin{array}{c}
      u_{n_2} \\
      v_{n_2} \\
    \end{array}
    \r)\r\|
    <C
    \l\|\l(
    \begin{array}{c}
      u_{n_1} \\
      v_{n_1} \\
    \end{array}
    \r)\r\|.
\end{equation}
Using the variation of parameters method one can rewrite the
system \eqref{simple system} in the following way:
\begin{multline}\label{variation}
    \left(%
        \begin{array}{c}
        u_{n+1} \\
        v_{n+1} \\
        \end{array}%
    \right)
    =\left(%
        \begin{array}{cc}
        \prod\limits_{k=n_0}^{n}(1-p_k) & 0 \\
        0 & 1 \\
        \end{array}%
    \right)
    \left(%
        \begin{array}{c}
        u_{n_0} \\
        v_{n_0} \\
        \end{array}%
    \right)
    \\
        + \sum\limits_{k=n_0}^{n}
        \left(%
            \begin{array}{cc}
            \prod\limits_{l=k+1}^n(1-p_l) & 0 \\
            0 & 1 \\
            \end{array}%
        \right) R_k
        \left(%
            \begin{array}{c}
            u_k \\
            v_k \\
            \end{array}%
        \right).
\end{multline}
The equivalence of the two systems \eqref{system} and
\eqref{variation} follows from elementary calculations. Let us
take
\begin{equation*}
    \left(%
        \begin{array}{c}
        u_{n_0} \\
        v_{n_0} \\
        \end{array}%
    \right)
    =
    \left(%
        \begin{array}{c}
        0 \\
        1 \\
        \end{array}%
    \right),
\end{equation*}
then
\begin{equation*}
    \left(%
        \begin{array}{c}
        u_{n+1} \\
        v_{n+1} \\
        \end{array}%
    \right)
    =
    \left(%
        \begin{array}{c}
        0 \\
        1 \\
        \end{array}%
    \right)
    + \sum\limits_{k=n_0}^{n}
    \left(%
        \begin{array}{cc}
        \prod\limits_{l=k+1}^n(1-p_l) & 0 \\
        0 & 1 \\
        \end{array}%
    \right) R_k
    \left(%
        \begin{array}{c}
        u_k \\
        v_k \\
        \end{array}%
    \right).
\end{equation*}
Note that $\prod\limits_{l=k+1}^n(1-p_l)\rightarrow0$ as $\ninf$
for every $k$ due to properties of the sequence $\{p_n\}$. Then
since $\l\|\l(
    \begin{array}{c}
      u_{n} \\
      v_{n} \\
    \end{array}
    \r)\r\|
<C$ for $n\ge n_0$ and $\{R_n\}\in l^1$, one has by Weierstrass
Theorem:
\begin{equation*}
    \sum\limits_{k=n_0}^{n}
    \left(%
        \begin{array}{cc}
        \prod\limits_{l=k+1}^n(1-p_l) & 0 \\
        0 & 1 \\
        \end{array}%
    \right) R_k
    \left(%
        \begin{array}{c}
        u_k \\
        v_k \\
        \end{array}%
    \right)
    \rightarrow
    \sum\limits_{k=n_0}^{\infty}
    \left(%
        \begin{array}{cc}
        0 & 0 \\
        0 & 1 \\
        \end{array}%
    \right) R_k
    \left(%
        \begin{array}{c}
        u_k \\
        v_k \\
        \end{array}%
    \right).
\end{equation*}
Hence,
\begin{equation*}
    \left(%
        \begin{array}{c}
        u_{n} \\
        v_{n} \\
        \end{array}%
    \right)
    \rightarrow
    \left(%
        \begin{array}{c}
        0 \\
        1 \\
        \end{array}%
    \right)
    +
    \sum\limits_{k=n_0}^{\infty}
    \left(%
        \begin{array}{cc}
        0 & 0 \\
        0 & 1 \\
        \end{array}%
    \right) R_k
    \left(%
        \begin{array}{c}
        u_k \\
        v_k \\
        \end{array}%
    \right)
\end{equation*}
as $\ninf$. The second component of the limit vector in the last
expression is non-zero provided that $n_0$ is sufficiently large
(due to \eqref{eq solutions boundedness} and the second condition
of Lemma \ref{new lem}).
\end{proof}

Reduction of system \eqref{system} to the special case of system
\eqref{simple system} can be done following the standard strategy
\cite{Coddington-Levinson}, \cite{Janas-Moszynski} by using the
fact of $D^1$-diagonalizability of every $D^1$ sequence of
matrices with invertible limit \cite{JM2} (also \cite[Lemma
1.3]{Janas-Moszynski}, this is a discrete version of the result
from \cite{Coddington-Levinson}):

\begin{prop}\cite{JM2}\label{diagonalizabilty_lemma}
    Let $\{V_n\}$ be a complex $2\times2$ matrix sequence such that $\{V_n\}\in D^1$ and $\rm{discr}\l(\lim\limits_{\ninf}V_n\r)\neq0$.
    Then the sequence $\{V_n\}$ is $D^1$-diagonalizable, i.e., there exists such a matrix sequence $\{T_n\}\in
    D^1$ with invertible limit that for sufficiently large values of
    $n$,
    \begin{equation*}
        V_n=T_n
        \left(%
            \begin{array}{cc}
            \mu_1(n) & 0 \\
            0 & \mu_2(n) \\
            \end{array}%
        \right) T_n^{-1}.
    \end{equation*}
\end{prop}

Let us return to the proof of Theorem \ref{JM lemma}. By
Proposition \ref{diagonalizabilty_lemma}, for $n$ sufficiently
large the corresponding matrices $T_n$ diagonalize $V_n$. To avoid
tedious notations, without loss of generality we will do all the
calculations starting with $n=1$. An explicit calculation shows:
\begin{multline*}
    T_{n+1}^{-1}(I+p_nV_n+R_n)T_n
    \\
        =(T_{n+1}^{-1}T_n)
        \l(I+p_n
        \left(%
            \begin{array}{cc}
            \mu_1(n) & 0 \\
            0 & \mu_2(n) \\
            \end{array}%
        \right)\r)+
        T_{n+1}^{-1}R_nT_n
        \\
            =I+p_n
            \left(%
                \begin{array}{cc}
                \mu_1(n) & 0 \\
                0 & \mu_2(n) \\
                \end{array}%
            \right)
            +Q_n,
\end{multline*}
where
\begin{equation*}
    Q_n:=
    T_{n+1}^{-1}(T_{n}-T_{n+1})
    \l[I+p_n
    \left(%
        \begin{array}{cc}
        \mu_1(n) & 0 \\
        0 & \mu_2(n) \\
        \end{array}%
    \right)\r]+T_{n+1}^{-1}R_nT_n.
\end{equation*}
Further,
\begin{multline*}
    \prod_{k=1}^{n}(I+p_kV_k+R_k)
        \\
            =T_{n+1}\l(\prod_{k=1}^{n}
            \l[I+p_k
            \left(%
                \begin{array}{cc}
                \mu_1(k) & 0 \\
                0 & \mu_2(k) \\
                \end{array}%
            \right)
            +Q_k\r]\r)T_1^{-1}
            \\
                =\l(
                \prod_{k=1}^{n}
                (1+p_k\mu_2(k))
                \r)
                T_{n+1}\l(
                \prod_{k=1}^{n}
                \l[I-
                \left(%
                    \begin{array}{cc}
                    \widetilde{p}_k & 0 \\
                    0 & 0 \\
                    \end{array}%
                \right)
                +\widetilde{R}_k\r]\r)T_1,
\end{multline*}
where
\begin{equation*}
    \widetilde{p}_n:=p_n\frac{\mu_2(n)-\mu_1(n)}{1+p_n\mu_2(n)},\ \widetilde{R}_n:=\frac{1}{1+p_n\mu_2(n)}Q_n.
\end{equation*}
The properties of the sequence $\{T_n\}$ guarantee that
$\{Q_n\}\in l^1$, as well as $\{\widetilde{R}_n\}\in l^1$.
Obviously $\widetilde{p}_n>0$ for large values of $n$, so the
system
\begin{equation*}
    \left(%
        \begin{array}{c}
        u_{n+1} \\
        v_{n+1} \\
        \end{array}%
    \right)
    =
    \l(I-
    \left(%
        \begin{array}{cc}
        \widetilde{p}_n & 0 \\
        0 & 0 \\
        \end{array}%
    \right)
    +\widetilde{R}_n\r)
    \left(%
        \begin{array}{c}
        u_{n} \\
        v_{n} \\
        \end{array}%
    \right)
\end{equation*}
satisfies all the conditions of Lemma \ref{new lem}. Therefore it has a solution $\left(%
\begin{array}{c}
  u_n \\
  v_n \\
\end{array}%
\right)=\overrightarrow{e}_2+o(1)$. Let
$T:=\lim\limits_{\ninf}T_n$. One has:
$T\overrightarrow{e}_2=\overrightarrow{x}_2$ is an eigenvector of
the matrix $V$ corresponding to the eigenvalue $\mu_2$. Then the
system \eqref{system} has a solution equal to
\begin{equation*}
    \l(\prod_{k=1}^{n}[1+p_k\mu_2(k)]\r)(\overrightarrow{x}_2+o(1))=:\overrightarrow{u}_n^{(2)},
\end{equation*}
which completes the proof.
\end{proof}

\subsection*{Acknowledgements}
The authors wish to express their gratitude to Dr. Marco Marletta
for his useful remarks and to the referee for reading of the
manuscript and making many important remarks.

\end{document}